\documentclass{amsart}
\usepackage[margin=1.5in]{geometry}
\usepackage{amsfonts}
\usepackage{lineno,hyperref}
\usepackage{amsmath}
\usepackage{amssymb}
\usepackage{amsthm}
\usepackage{mathtools}
\usepackage{cancel}
\usepackage{tikz-cd}
\usepackage{arydshln}
\usepackage{graphicx}
\usepackage{float}
\usepackage{algorithm}
\usepackage[noend]{algpseudocode}
\usepackage{thmtools}
\usepackage{thm-restate}
\usepackage{hyperref}
\usepackage{cleveref}

\newcommand{\RR}{\mathbb{R}}

\newcommand{\NN}{\mathbb{N}}

\newcommand{\ZZ}{\mathbb{Z}}
\newcommand{\QQ}{\mathbb{Q}}

\theoremstyle{plain}
\newtheorem{definition}{Definition}[section]
\newtheorem{theorem}{Theorem}[section]
\newtheorem{conjecture}[theorem]{Conjecture}

\newtheorem{lemma}[theorem]{Lemma}

\theoremstyle{remark}

\newtheorem{alg}{Algorithm}[section]
\newtheorem{remark}{Remark}[section]

\begin{document}
\title{The Ulam Sequence of the Integer Polynomial Ring}
\author{Arseniy (Senia) Sheydvasser}
\address{92 Mount Vickery Road, Southborough, MA 01772}
\email{sheydvasser@gmail.com}


\date{\today}


\begin{abstract}
An Ulam sequence $U(1,n)$ is defined as the sequence starting with integers $1,n$ such that $n > 1$, and such that every subsequent term is the smallest integer that can be written as the sum of distinct previous terms in exactly one way. This family of sequences is notable for being the subject of several remarkable rigidity conjectures. We introduce an analogous notion of an Ulam sequence inside the polynomial ring $\ZZ[X]$, and use it both to give new, constructive proofs of old results as well as producing a new conjecture that implies many of the other existing conjectures.
\end{abstract}

\maketitle

\section{Introduction and Main Results:}\label{section: introduction}

\begin{figure}
\centering
\includegraphics[width=0.7\textwidth]{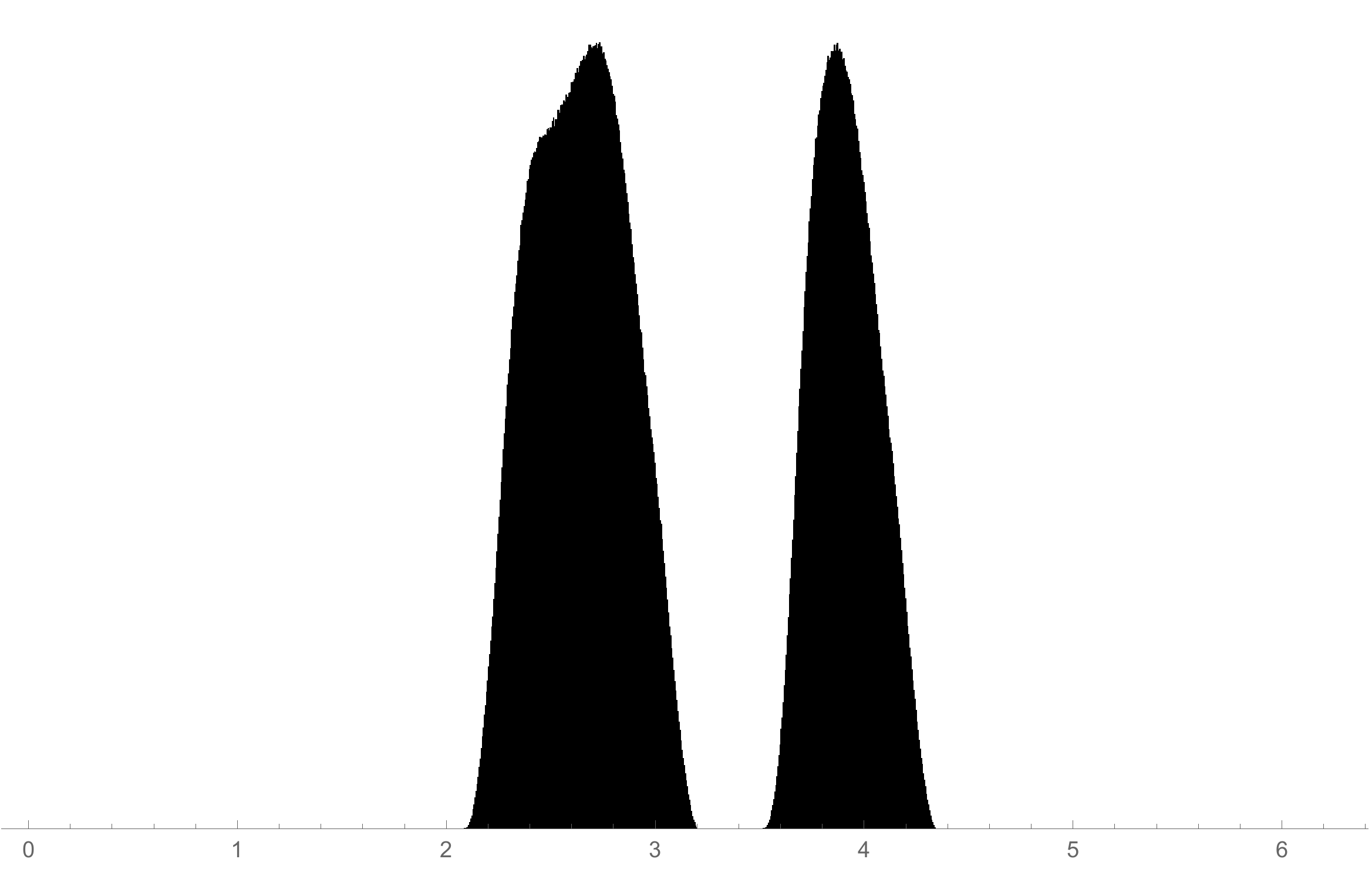}
\caption{Histogram of the first billion terms of $U(1,2) \mod \lambda_2$, rescaled into the interval $[0,2\pi]$.}
\label{First Histogram}
\end{figure}

Given integers $1 \leq a < b$, define the \emph{Ulam sequence} $U(a,b)$ to be the sequence starting with $a,b$, and such that every subsequent term is the smallest integer that can be written as the sum of two distinct prior terms in exactly one way. The sequence
	\begin{align*}
	U(1,2) &= 1, 2, 3, 4, 6, 8, 11, 13, 16, 18, 26, 28, 36, 38, 47, 48, 53, 57, 62, 69\ldots
	\end{align*}
	
\noindent was originally introduced in 1964 by Ulam \cite{ulam_1964}, who posed the question of determining the growth rate of this sequence, which remains opens to this day. It is conjectured that $U(1,2)$ grows linearly, and it has positive density of about $0.079$. The growth rate of certain other families of Ulam sequences was confirmed to be linear by proving the stronger result that they are eventually periodic---this was done for Ulam sequences $U(2,2n + 1)$ by Schmerl and Spiegel \cite{schmerl_spiegel_1994}, using previous work of Finch \cite{finch_1991,finch_1992_1,finch_1992_2} and Queneau \cite{queneau_1972}. Similarly, Cassaigne and Finch \cite{cassaigne_finch_1995} proved that sequences $U(4,n)$ are eventually periodic if $n \equiv 1 \mod 4$, and Hinman, Kuca, Schlesinger, and Sheydvasser \cite{HKSS_2019_2} found a finite set of sequences $U(4,n)$ with $n \equiv 3 \mod 4$ that are also eventually periodic. In contrast, none of the sequences of the form $U(1,n)$ seem to be eventually periodic, and virtually nothing was known about them until very recently, when Steinerberger \cite{steinerberger_2016} gave numerical evidence that there exists a real number $\lambda_2 \approx 2.443442967784743433$ with the curious property that $U(1,2) \mod \lambda_2$ is concentrated in the middle third of the interval. To be more precise, we have the following conjecture, formulated by Gibbs \cite{Gibbs_2015}.
	\begin{conjecture}\label{conj:gibbs}
	There exists a real number $\lambda_2 \approx 2.443442967784743433$ such that for all $\epsilon > 0$, for $K$ sufficiently large,
		\begin{align*}
		U(1,2) \cap [K,\infty) \mod \lambda_2 \subset \left(\frac{\lambda_2}{3} - \epsilon, \frac{2\lambda_2}{3} + \epsilon\right).
		\end{align*}
	\end{conjecture}
	
\noindent This conjecture has since been confirmed for the first trillion terms of $U(1,2)$ by Gibbs and McCranie \cite{Gibbs_Judson_2015}. Numerical evidence suggests that for Ulam sequences $U(a,b)$ that are not eventually periodic, similar behavior occurs---such ``magic numbers" for Ulam sequences are referred to as \emph{periods} in the literature. In particular, there is the following generalization of Gibbs' conjecture in the mathematical folklore\footnote{It was communicated to the author by Joshua Hinman, who did an extensive numerical study of periods of various families of Ulam sequences.}.

\begin{conjecture}\label{conj:gibbslike}
For all $n \geq 2$, there exist real numbers $\lambda_n$, $K_n$ such that for all $\epsilon > 0$,
	\begin{align*}
	U(1,n) \cap [K_n,\infty) \mod \lambda_n \subset \left(\frac{\lambda_n}{3} - \epsilon, \frac{2\lambda_n}{3} + \epsilon\right).
	\end{align*}
	
\noindent Furthermore, for $n \geq 4$, we can take $\lambda_n = 3n + \lambda'$, where $\lambda' \approx 0.417031$.
\end{conjecture}

\noindent The observed empirical fact that for $n \geq 4$, the periods $\lambda_n$ grow linearly has been poorly understood up until now. Our goal is to show that this curious phenomenon is deeply tied to the following---seemingly unconnected---numerical observation of Hinman, Kuca, Schlesinger, and the present author \cite{HKSS_2019_1}. Specifically, we noted that for $n \geq 4$, runs of consecutive elements of $U(1,n)$ group into blocks whose endpoints grow linearly.
	\begin{align*}
	\begin{array}{rllllll}
	U(1,4)= & \boxed{1}, & \boxed{4, 5, 6, 7, 8}, & \boxed{10}, & \boxed{16}, & \boxed{18, 19}, & \boxed{21} \ldots \\
	U(1,5)= & \boxed{1}, & \boxed{5, 6, 7, 8, 9, 10}, & \boxed{12}, & \boxed{20}, & \boxed{22, 23, 24}, & \boxed{26} \ldots \\
	U(1,6)= & \boxed{1}, & \boxed{6, 7, 8, 9, 10, 11, 12}, & \boxed{14}, & \boxed{24}, & \boxed{26, 27, 28, 29}, &\boxed{31} \ldots \\
	U(1,n)= & \boxed{1}, & \boxed{n, n + 1, \ldots 2n}, & \boxed{2n + 2}, & \boxed{4n}, & \boxed{4n + 2 \ldots 5n - 1}, & \boxed{5n + 2} \ldots
	\end{array}
	\end{align*}
	
\noindent This observation was made precise by the following conjecture.

\begin{conjecture}\label{conj:rigidity}
There exist unique integer coefficients $a_i,b_i,c_i,d_i$ such that for all $n \geq 4$,
	\begin{align*}
	U(1,n) = \bigcup_{i = 0}^\infty [a_i n + b_i, c_i n + d_i],
	\end{align*}
	
\noindent such that $a_{i + 1}n + b_{i + 1} > c_i n + d_i n + 1$ for all $i$.
\end{conjecture}

\noindent At present, this conjecture is still wide open, but there is a somewhat weaker result.
	\begin{theorem}[Theorem 1.1 of \cite{HKSS_2019_1}]\label{Weak Rigidity Theorem}
	There exist integer coefficients $a_i, b_i, c_i, d_i$ such that for any $C > 0$, there exists a positive integer $N$ such that for all integers $n \geq N$,
		\begin{align*}
    U(1,n) \cap [1,Cn] = \left(\bigcup_{i = 0}^\infty [a_i n + b_i, c_i n + d_i] \right) \cap [1,Cn].
    \end{align*}
	\end{theorem}
	
\begin{table}
\renewcommand{\arraycolsep}{1.8pt}
\begin{align*}
\begin{array}{lllllllllll}
U(1,X) = \{1\} & \cup & [X, 2 X] & \cup & \{2 X+2\} & \cup & \{4 X\} \\
& \cup & [4 X+2, 5 X-1] & \cup & \{5 X+1\} & \cup & [7 X+3, 8 X+1] \\
& \cup & \{10 X+2\} & \cup & \{11 X+2\} & \cup & [13 X+4, 14 X+1] \\
& \cup & \{16 X+2\} & \cup & \{17 X+2\} & \cup & \{19 X+3\} \\
& \cup & \{20 X+2\} & \cup & \{22 X+3\} & \cup & \{23 X+4\} \\
& \cup & [25 X+4, 25 X+5] & \cup & \{26 X+3\} & \cup & \{28 X+4\} \\
& \cup & [31 X+5, 32 X+3] & \cup & \{34 X+5\} & \cup & \{38 X+6\} \\
& \cup & \{40 X+5\} & \cup & [40 X+8, 41 X+4] & \cup & [43 X+7, 44 X+4] \\
& \cup & \{44 X+6\} & \cup & \{46 X+7\} & \cup & [49 X+8, 50 X+6] \\
& \cup & [52 X+8, 53 X+7] & \cup & [55 X+9, 56 X+6] \ldots
\end{array}
\end{align*}

\caption{The first $30$ intervals of $U(1,X)$---equivalently, replacing $X$ with $n$, the first $56n + 6$ terms of $U(1,n)$ if $n \geq 4$.}
\label{tab:First 30 Intervals}
\end{table}

\noindent The original proof of this theorem was nonconstructive and in fact model theoretic in nature, although it was shown in \cite{HKSS_2019_1} that assuming Theorem \ref{Weak Rigidity Theorem}, one can prove that there exists an algorithm that will both find the coefficients $a_i,b_i,c_i,d_i$ and the minimal integer $N_0$ that will satisfy the conditions of Theorem \ref{Weak Rigidity Theorem}, given $C > 0$. Our present goal is to give an alternate, constructive proof of Theorem \ref{Weak Rigidity Theorem} by considering a variant of the Ulam sequence that sits inside the polynomial ring $\ZZ[X]$. This ring can be given the structure of an ordered ring by giving it the lexicographical ordering---that is, $p(X) > q(X)$ if and only if the leading term of $p(X) - q(X)$ has a positive coefficient. In Section \ref{section: new Ulam sequence}, we define a set
	\begin{align*}
	U(1,X) = \bigcup_{i = 0}^\infty [a_i X + b_i, c_i X + d_i] \subset \ZZ[X]
	\end{align*}
	
\noindent which should be viewed as an analog of an Ulam sequence inside this ordered ring.  Here $[x,y]$ has the usual meaning that it is the set of all elements $z \in \ZZ[X]$ such that $x \leq z \leq y$. We call the sequences $\{a_i\}_{i \in \NN}$, $\{b_i\}_{i \in \NN}$, $\{c_i\}_{i \in \NN}$, and $\{d_i\}_{i \in \NN}$ the \emph{coefficients} of $U(1,X)$---as we shall see, they are uniquely determined and, better yet, there is an algorithm to compute them.

\begin{theorem}\label{Main Theorem}
There exists a polynomial-time algorithm $\mathcal{A}_\text{Ulam}$ such that on an input of $k \in \NN$, it returns the first $k$ coefficients of $U(1,X)$---that is, $\{a_i\}_{i = 0}^k$, $\{b_i\}_{i = 0}^k$, $\{c_i\}_{i = 0}^k$, and $\{d_i\}_{i = 0}^k$---and an integer $N$ such that for all $n \geq N$,
	\begin{align*}
	U(1,n) \cap [1, c_k n + d_k] = \bigcup_{i = 0}^\infty [a_i n + b_i, c_i n + d_i] \cap [1, c_k n + d_k].
	\end{align*}
\end{theorem}

\noindent Theorem \ref{Weak Rigidity Theorem} is an immediate corollary of Theorem \ref{Main Theorem}. We construct an explicit example of such an algorithm in Section \ref{section: algorithms} and show that there should be many other such algorithms, including ones that are likely much more efficient. Studying the output of such algorithms raises the possibility of proving that
	\begin{align*}
	U(1,n) = \bigcup_{i = 0}^\infty [a_i n + b_i, c_i n + d_i]
	\end{align*}
	
\noindent for all $n \geq N$ for some natural number $N$. To the best of the author's knowledge, this is the first proposed method of attacking Conjecture \ref{conj:rigidity}. However, this is not the only benefit of introducing the set $U(1,X)$---it also makes it convenient to state a conjecture for which the author found ample numerical evidence.

\begin{conjecture}\label{conj:stat}
Let $a_i,b_i,c_i,d_i$ be the coefficients of $U(1,X)$. There exist real numbers $\lambda' \approx 0.417031, \sigma_1 \approx 1.86, \sigma_2 \approx -1.3$ such that for any $\epsilon > 0$, if $i$ is sufficiently large, then
	\begin{align*}
	a_iX + b_i, c_iX + d_i \mod 3X + \lambda' \in &\left(X + \frac{\lambda'}{3} - \epsilon, X + \sigma_1 + \epsilon\right) \\ \cup &\left(2X + \sigma_2 - \epsilon, 2X + \frac{2\lambda'}{3} + \epsilon\right).
	\end{align*}
\end{conjecture}

\noindent The precise definition of taking a modulus in $\ZZ[X]$ shall be given in Section \ref{section: numerical}. This conjecture should be seen as an analog of Conjecture \ref{conj:gibbslike} for the ordered ring $\ZZ[X]$. In Section \ref{section: conjectures}, we also demonstrate that Conjecture \ref{conj:stat} has a number of remarkable consequences---for example, it implies that $b_i$ grows linearly with respect to $a_i$, and similarly $d_i$ grows linearly with respect to $c_i$; furthermore, we prove that Conjectures \ref{conj:stat} and \ref{conj:rigidity} together imply Conjecture \ref{conj:gibbslike} for $n \geq 4$.

\begin{figure}
\begin{center}
\includegraphics[width=0.85\textwidth]{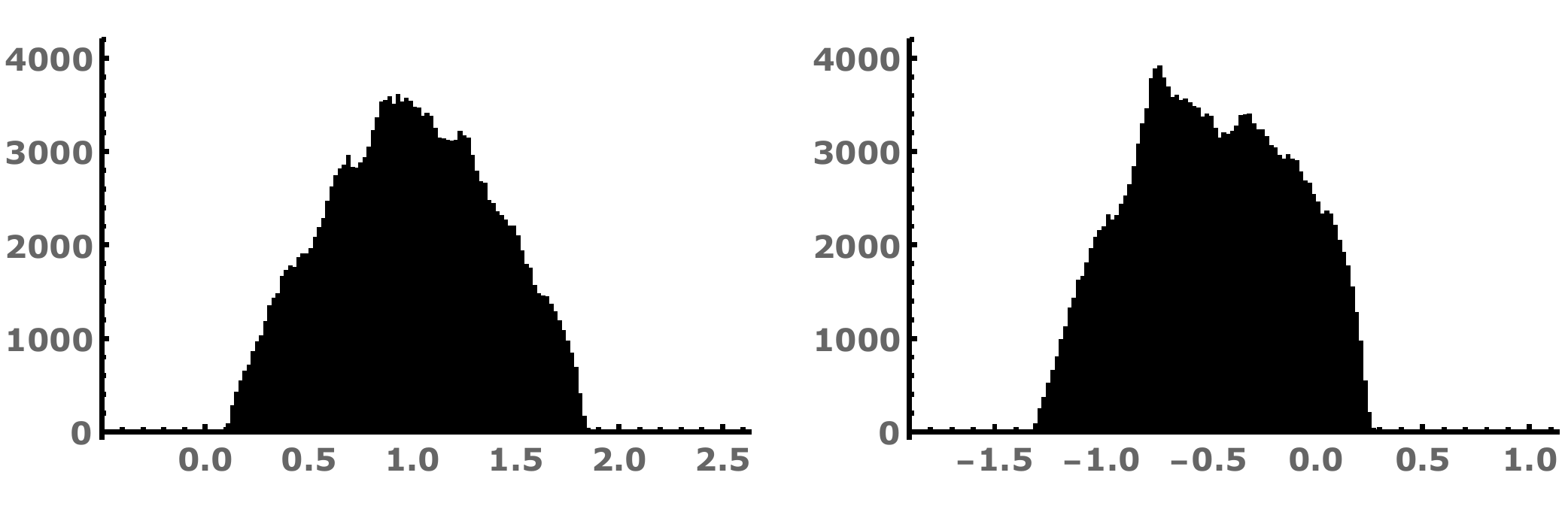}
\end{center}

\caption{Histogram of $U(1,X) \mod \lambda(X)$ for the first 200000 intervals. The plot on the left represents elements in $(X + \frac{\lambda'}{3} - \epsilon, X + \sigma_1 + \epsilon)$---the plot on the right represents elements in $(2X + \sigma_2 - \epsilon, 2X + \frac{2\lambda'}{3} + \epsilon)$.}
\label{Ulam Statistics}
\end{figure}

\subsection*{Acknowledgements:}

The author is exceptionally grateful to Oleg Sheydvasser, for lending his expertise in implementing efficient versions of the algorithms discussed in this paper, and for lending his processor time to run them. The author would also like to thank Stefan Steinerberger for providing helpful discussion and feedback, and Camp Minnetonka for the atmosphere in which this paper was finished.
	
\section{An Ulam-Like Sequence:}\label{section: new Ulam sequence}

We begin by defining what we mean by Ulam sequences inside of the ordered ring $\ZZ[X]$.

\begin{definition}
Write $G_d$ for the additive subgroup of $\ZZ[X]$ generated by $1,X,\ldots,X^d$; this is automatically an ordered group. Given $0 < a < b$ in $\ZZ[X]$, an \emph{Ulam sequence} starting with $a,b$ is a set $U \subset G_{\text{deg}(b)}$ such that
    \begin{enumerate}
        \item $U \cap (-\infty,b] = \{a,b\}$,
        \item for all $p < q \in G_d$, $U \cap [p,q]$ has both a minimum and a maximum, and
        \item for every $p \in (b,\infty)$, $p \in U$ if and only if it is the smallest element in the set
            \begin{align*}
                \left\{q \in G_{\text{deg}(b)}\middle|q > U \cap (-\infty,p) \text{ and } \exists! x \neq y \in U, \ q = x + y\right\}.
            \end{align*}
    \end{enumerate}
\end{definition}

\begin{table}
    \centering
    \begin{align*}
        X,&X+1,2 X+1,3 X+1,3 X+2,4 X+1,4 X+3,5 X+1,5 X+4,6 X+1, \\
        &6 X+3,6 X+5,7 X+1,7 X+6,8 X+1,8 X+3,8 X+5,8 X+7,9 X+1, \\
        &9 X+8,10 X+1,10 X+3,10 X+5,10 X+7,10 X+9,11 X+1,11 X+10,\\
        &12 X+1,12 X+3,12 X+5,12 X+7,12 X+9,12 X+11,13 X+1,13 X+12, \\
        &14 X+1,14 X+3,14 X+5,14 X+7,14 X+9,14 X+11,14 X+13,15 X+1,\\
        &15 X+14,16 X+1,16 X+3,16 X+5,16 X+7,16 X+9,16 X+11,16 X+13,\\
        &16 X+15,17 X+1,17 X+16,18 X+1,18 X+3,18 X+5,18 X+7,18 X+9,\\
        &18 X+11,18 X+13,18 X+15,18 X+17,19 X+1,19 X+18,20 X+1,20 X+3,\\
        &20 X+5,20 X+7,20 X+9,20 X+11,20 X+13,20 X+15,20 X+17,20 X+19,\\
        &21 X+1,21 X+20,22 X+1,22 X+3,22 X+5,22 X+7,22 X+9,22 X+11,\\
        &22 X+13,22 X+15,22 X+17,22 X+19,22 X+21,23 X+1,23 X+22,\\
        &24 X+1,24 X+3,24 X+5,24 X+7,24 X+9,24 X+11,24 X+13,\\
        &24 X+15,24 X+17,24 X+19 \ldots
    \end{align*}
    \caption{The first 100 terms of the unique Ulam sequence starting with $X,X+1$.}
    \label{tab:weird ulam sequence}
\end{table}

\begin{remark}
If $a,b \in \ZZ$, then this just reduces to the usual definition of an Ulam sequence.
\end{remark}

\begin{remark}
The idea of extending Ulam sequences to other ordered abelian groups than the integers is not new---Kravitz and Steinerberger previously showed that one can define the notion of an Ulam subset of $\RR^n$ \cite{kravitz_steinerberger_2017}.
\end{remark}

Why do we say that it is \emph{an} Ulam sequence, rather than \emph{the} Ulam sequence? This is because, in general, it is not unique! For example, it follows from earlier model theoretic results \cite{HKSS_2019_1} that there exist Ulam sequences $U,U'$ starting with $2,X$ such that
    \begin{align*}
        U \cap [2,2X + 1] &= \{2\} \cup \{X + 2k|k \in \NN\} \cup \{2X - 2k|k \in \NN\} \\
        U' \cap [2,2X + 1] &= \{2\} \cup \{X + 2k|k \in \NN\} \cup \{2X - 2k + 1|k \in \NN\}.
    \end{align*}
    
\noindent Nevertheless, there are some specific cases where such an Ulam sequence is unique---in particular, we shall show that there is a unique Ulam sequence starting with $1,X$, which we shall denote by $U(1,X)$. This set is very special in that it is a countable union of disjoint intervals that appear in order.

\begin{definition}
We say that $S \subset G_1$ is a \emph{$DS$-subset}, if there exists some integer sequences $a_i,b_i,c_i,d_i$ such that
	\begin{align*}
	S = \bigcup_{i \in I} \left[a_i X + b_i, c_i X + d_i\right]
	\end{align*}
	
\noindent where $I$ is a subset of $\ZZ$, $c_i X + d_i + 1 < a_{i + 1}  + b_{i + 1}$ for all $i \in I$, and for all $i < i' \in I$, if $u \in [a_i X + b_i, c_i X + d_i]$ and $v \in [a_{i'}X + b_{i'}, c_{i'} X + d_{i'}]$, then $u < v$. We call the sequences $a_i,b_i,c_i,d_i$ the \emph{coefficients} of $S$.
\end{definition}

It is easy to determine the starting intervals of this set.

\begin{lemma}
Let $U$ be an Ulam sequence starting with $1,X$. Then $U \cap [1,2X + 1] = \{1\} \cup [X,2X]$.
\end{lemma}

\begin{proof}
Obviously, $U$ contains both $1$ and $X$, hence it contains $X + 1$. By induction, it must also contain $X + k$ for all $k \in \NN$. It follows that it must contain elements $2X - k$ with $k$ arbitrarily large, since otherwise $U \cap [1,2X - k]$ would fail to have a maximum for some $k \in \NN$. But $2X-k$ must be the sum of two prior elements $u < v$, and if $u,v \geq X$, then it would have to be at least $2X + 1 = X + (X + 1)$. Therefore, $u = 1$ and $v = 2X - k - 1$, which is to say that $[X,2X] \subset U$. Finally, since $2X + 1$ can be written as the sum of distinct elements in $U$ in two different ways, $2X + 1 \notin U$.
\end{proof}

\begin{theorem}
There exists a unique Ulam sequence $U(1,X)$ starting with $1,X$. Furthermore, $U(1,X)$ is a $DS$-subset.
\end{theorem}

\begin{proof}
We shall prove by induction that for all $k \in \NN$, there exists $l_k \in \ZZ$ and a unique set $U \subset [1,kX + l_k]$ such that
    \begin{enumerate}
        \item $U \cap [1,X] = \{1,X\}$,
        \item $U$ is a finite union of intervals,
        \item for every $u \in [1,kX + l_k]$, $u \in U$ if and only if there exists a unique pair $p < q \in U$ such that $u = p + q$.
    \end{enumerate}
    
\noindent We already know that this is true for $k = 0,1$. So, assume that it is true for $k - 1$---we shall prove it for $k$. Start with $U \subset [1,(k - 1)X + l_{k - 1}]$ which has the desired properties. By induction, there is a unique extension $U' \subset [1,(k - 1)X + l]$ for all $l \in \NN$ such that
    \begin{enumerate}
        \item $U' \cap [1,X] = \{1,X\}$ and
        \item for every $u \in [1,(k - 1)X + l]$, $u \in U'$ if and only if there exists a unique pair $p < q \in U'$ such that $u = p + q$.
    \end{enumerate}
    
\noindent Moreover, our claim is that if $l$ is sufficiently large, then either $U' \cap [(k - 1)X + l,(k - 1)X + l'] = [(k - 1)X + l,(k - 1)X + l']$ for all larger $l'$, or $U' \cap [(k - 1)X + l,(k - 1)X + l'] = \emptyset$. This is because if $u \in U'$, then $u + 1 \in U'$ unless there exist $X \leq p < q$ such that $u + 1 = p + q$. But it must be that $p < q < (k - 2)X + l$ for some $l \in \ZZ$ since otherwise $p + q$ would be too big. Therefore, $p,q$ both belong to one of the intervals of $U$---there are only finitely many of these, hence
    \begin{align*}
        [1,kX] \backslash \left\{p + q\middle|1 < p < q \in U\right\}
    \end{align*}
    
\noindent is a finite union of intervals. Therefore, either if $l$ is sufficiently large the intersection of this set with $[(k - 1)X + l, kX - l]$ is either all of $[(k - 1)X + l, kX - l]$ or empty. In the first case, $U' \cap [(k - 1)X + l,(k - 1)X + l'] = \emptyset$; in the second case, either $U' \cap [(k - 1)X + l,(k - 1)X + l'] = \emptyset$ or $U' \cap [(k - 1)X + l,(k - 1)X + l'] = [(k - 1)X + l,(k - 1)X + l']$, depending on whether $(k - 1)X + l \in U'$.

If there exists $l \in \NN$ such that $U' \cap [(k - 1)X + l,(k - 1)X + l'] = [(k - 1)X + l,(k - 1)X + l']$ for all larger $l'$, then there exists some $l'' \in \ZZ$ such that
    \begin{align*}
        \left([1,kX] \backslash \left\{p + q\middle|1 < p < q \in U\right\}\right) \cap [(k - 1)X + l, kX - l''] = \emptyset.
    \end{align*}
    
\noindent Therefore, if we define $l_k = l''$, $U'' = U \cup [(k - 1)X + l, kX + l'']$, then
    \begin{enumerate}
        \item $U'' \cap [1,X] = \{1,X\}$,
        \item $U''$ is a finite union of intervals,
        \item for every $u \in [1,kX + l_k]$, $u \in U''$ if and only if there exists a unique pair $p < q \in U''$ such that $u = p + q$.
    \end{enumerate}
    
\noindent Moreover, this is the only possible extension of $U'$ to $[1,kX + l_k]$, since the requirement that $U'' \cap [1,kX + l_k]$ have a maximum requires that there must be some element $kX + l''' \in U''$---but this forces $kX + l''' - 1, kX + l'''-2,\ldots \in U''$ and $kX + l''' + 1, kX + l''' + 2,\ldots kX + l_k \in U$.

Otherwise, $U' \cap [(k - 1)X + l,kX - l] = \emptyset$ for sufficiently large $l$. We can choose $l$ large enough such that for all $p \in [(k - 1)X + l, kX - l]$, either there exists more than one pair $X \leq u,v < kX - l$ such that $p = u + v$, or there are no such pairs. Then if we define $l_k = -l$, then $U'$ is a subset of $[1,kX + l_k$ such that
    \begin{enumerate}
        \item $U' \cap [1,X] = \{1,X\}$,
        \item $U'$ is a finite union of intervals,
        \item for every $u \in [1,kX + l_k]$, $u \in U'$ if and only if there exists a unique pair $p < q \in U'$ such that $u = p + q$.
    \end{enumerate}
    
\noindent Moreover, this is the only possible way to extend $U'$ to $[1,kX + l_k]$---if we were to add any element $kX + l$, it would require that $kX + l - 1, kX + l - 2, \ldots \in U'$, which would contradict the requirement that $U' \cap [(k - 1)X + l,kX + l_k]$ has a minimum.

Thus, we have shown that there is a unique $U \subset [1,kX + l_k]$ for all $k$. But this gives a unique $U \subset G_1$ which is just a union of all these sets---it is easy to see that this $U$ is the unique Ulam sequence starting with $1,X$, and that it is a $DS$-subset.
\end{proof}

\section{Algorithms:}\label{section: algorithms}

We have shown that $U(1,X)$ is uniquely determined and a highly structured set, but this is not all: its coefficients are computable, and the algorithm that does it can be used to give a very nice proof of Theorem \ref{Main Theorem}. Before we present this, we will need two intermediate subroutines.

\begin{alg}\label{Sum Algorithm}
On an input of two intervals $I_1 = [p_1,q_1]$, $I_2 = [p_2,q_2]$ such that $p_2 + 1 < q_1$, this algorithm returns a pair of $DS$-subsets $S_1, S_2$ such that $S_1$ consists of all elements that can be written as the sum of an element of $I_1$ and an element of $I_2$ in exactly one way, and $S_2$ consists of all elements that can be written as the sum of an element of $I_1$ and an element of $I_2$ in more than one way.
\begin{algorithm}[H]
\begin{algorithmic}[1]
\Procedure{Sum1}{$I_1$,$I_2$}
\State $start \gets p_1 + p_2$
\State $end \gets q_1 + q_2$
\If{$\# I_1 = 1$ or $\# I_2 = 1$}
	\State $S_1 \gets [start, end]$
	\State $S_2 \gets \{\}$
\Else
	\State $S_1 \gets \{start\} \cup \{end\}$
	\State $S_2 \gets [start + 1, end - 1]$
\EndIf
\State \textbf{return} $S_1, S_2$
\EndProcedure
\end{algorithmic}
\end{algorithm}
\end{alg}

\begin{proof}[Proof of Correctness]
Clearly, $I_1 + I_2 = [start, end]$, so it is solely a question of how this set is partitioned between $S_1$ and $S_2$. If $\# I_1 = 1$ or $\# I_2 = 1$, then it is easy to see that $S_2$ is empty and $S_1 = [start, end]$. Otherwise, $start, end \in S_1$, but for any $x \in [start + 1, end - 1]$ we can write $x = u + v = (u - 1) + (v + 1)$ for some $u,(u - 1) \in I_1$, $v, (v + 1) \in I_2$.
\end{proof}

The importance of this algorithm will come when we have to look at the sums of all of the elements in two distinct intervals in the Ulam sequence. However, we shall also have to consider sums of all elements in one given interval, and that is handled by our second subroutine.

\begin{alg}\label{Sum Algorithm Redux}
On an input of an interval $I = [p,q]$, this algorithm returns a pair of $DS$-subsets $S_1, S_2$ such that $S_1$ consists of all elements that can be written as the sum of two distinct elements of $I$ in exactly one way, and $S_2$ consists of all elements that can be written as the sum of two distinct elements of $I$ in more than one way.
\begin{algorithm}[H]
\begin{algorithmic}[1]
\Procedure{Sum2}{$I$}
\If{$\# I = 1$}
	\State $S_1 \gets \{\}$
	\State $S_2 \gets \{\}$
\ElsIf{$\# I = 2$}
	\State $S_1 \gets \{p + q\}$
	\State $S_2 \gets \{\}$
\ElsIf{$\# I = 3$}
	\State $S_1 \gets [2p + 1, 2p + 3]$
	\State $S_2 \gets \{\}$
\Else
	\State $S_1 \gets [2p + 1, 2p + 2] \cup [2q - 2, 2q - 1]$
	\State $S_2 \gets [2p + 3, 2q - 3]$
\EndIf
\State \textbf{return} $S_1, S_2$
\EndProcedure
\end{algorithmic}
\end{algorithm}
\end{alg}

\begin{proof}[Proof of Correctness]
It is clear that the subset of $\ZZ[X]$ representable by pairwise sums of distinct elements of $I$ is $[2p + 1, 2q - 1]$. It is easy to see that $2p + 1, 2p + 2, 2q - 2, 2q - 1 \in S_1$ if they are in this subset, while the remainder must be in $S_2$.
\end{proof}

Finally, we can describe the algorithm for computing the coefficients of $U(1,X)$.

\begin{alg}\label{Naive Algorithm}
On an input of a natural number $k$, this algorithm returns the first $k + 1$ coefficients $a_i,b_i,c_i,d_i$ of $U(1,X)$. This algorithm keeps track of the following three $DS$-subsets.
	\begin{enumerate}
		\item $ulam\_ds$ maintains the subset of $U(1,X)$ computed thus far.
		\item $one\_rep\_ds$ maintains a subset of $\ZZ[X]$ such that every element of $one\_rep\_ds$ is larger than every element of $ulam\_ds$, and each element of $one\_rep\_ds$ can be written as a sum of distinct elements in $ulam\_ds$ in exactly one way.
		\item $mult\_rep\_ds$ maintains a subset of $\ZZ[X]$ such that every element of $mult\_rep\_ds$ is larger than every element of $ulam\_ds$, and each element of $mult\_rep\_ds$ can be written as a sum of distinct elements in $ulam\_ds$ in more than one way.
	\end{enumerate}
\begin{algorithm}[H]
\begin{algorithmic}[1]
\Procedure{UlamCoefficients}{$k$}
\State $ulam\_ds \gets \{1\} \cup [X,2X]$
\State $one\_rep\_ds \gets \{\}$
\State $mult\_rep\_ds \gets \{2X + 1\}$
\State $largest\_computed \gets 2X$
\For{$1 < i \leq k$}
	\State $last \gets$ last interval in $ulam\_ds$
	\For{$interval \in ulam\_ds$}
		\If{$interval = last$}
			\State $(one\_rep\_guess\_ds, mult\_rep\_guess\_ds) \gets \text{Sum2}(last)$
		\Else
			\State $(one\_rep\_guess\_ds, mult\_rep\_guess\_ds) \gets \text{Sum1}(interval,last)$
		\EndIf
		\State $one\_rep\_guess\_ds \gets one\_rep\_guess\_ds \cap [largest\_computed + 2, \infty)$
		\State $mult\_rep\_guess\_ds \gets mult\_rep\_guess\_ds \cap [largest\_computed + 2, \infty)$
		\State $one\_rep\_ds \gets one\_rep\_ds \backslash \left(one\_rep\_ds \cap mult\_rep\_guess\_ds\right)$
		\State $one\_rep\_guess\_ds \gets one\_rep\_guess\_ds \backslash \left(one\_rep\_guess\_ds \cap mult\_rep\_ds\right)$
		\State $temp\_ds \gets$ the symmetric difference of $one\_rep\_guess\_ds$ and $one\_rep\_ds$
		\State $mult\_rep\_additional\_ds \gets \left(one\_rep\_ds \cup one\_rep\_ds\right) - temp\_ds$
		\State $one\_rep\_ds \gets temp\_ds$
		\State $mult\_rep\_ds \gets mult\_rep\_ds \cup mult\_rep\_guess\_ds \cup mult\_rep\_additional\_ds$
	\EndFor
	
	\State $[p,q] \gets$ smallest interval in $one\_rep\_ds$
	\If{$p = q$}
		\State $bound \gets p + X$
		\State $one\_rep\_bound \gets \min\left(bound, \min_{x > p}\left(x \in one\_rep\_ds\right)\right)$
		\State $mult\_rep\_bound \gets \min\left(bound, \min_{x > p}\left(x \in mult\_rep\_ds\right)\right)$
		\State $bound \gets \min\left(one\_rep\_bound,mult\_rep\_bound\right)$
		\State $new\_interval = [p, bound - 1]$
	\Else
		\State $new\_interval = [p,p]$
	\EndIf
	\State $ulam\_ds \gets ulam\_ds \cup new\_interval$
	\State $largest\_computed \gets \max\left(ulam\_ds\right)$
	\State $one\_rep\_ds \gets one\_rep\_ds \cap [largest\_computed + 2, \infty)$
	\State $mult\_rep\_ds \gets mult\_rep\_ds \cap [largest\_computed + 2, \infty)$
\EndFor
\State \textbf{return} coefficients of $ulam\_ds$
\EndProcedure
\end{algorithmic}
\end{algorithm}
\end{alg}

\begin{proof}[Proof of Correctness]
For any $l \in \NN$, let $\mathcal{U}_l$ consist of the first $l + 1$ intervals of $U(1,X)$. We shall show that at the end of each cycle of the outer for-loop indexed over $i$, $ulam\_ds = \mathcal{U}_i$, $largest\_computed$ is the largest element of $ulam\_ds$, $one\_rep\_ds$ consists of all elements of $\mathcal{U}_{i - 1} + \mathcal{U}_{i - 1}$ larger than $largest\_computed + 1$ that can be written as a sum of two distinct elements in $\mathcal{U}_{i - 1}$ in exactly one way, and $mult\_rep\_ds$ consists of all elements of $\mathcal{U}_{i - 1} + \mathcal{U}_{i - 1}$ larger than $largest\_computed + 1$ that can be written as a sum of two distinct elements in $\mathcal{U}_{i - 1}$ in more than one way. We prove this claim by induction on $i$.

The base case $i = 1$ is obvious---this is just the initialization of $ulam\_ds$, $largest\_computed$, $one\_rep\_ds$, and $mult\_rep\_ds$ prior to the for-loop. For all subsequent $i$, note that in order to find all elements $\mathcal{U}_{i - 1} + \mathcal{U}_{i - 1}$ that need to be added to $one\_rep\_ds$ and $mult\_rep\_ds$, it suffices to consider the sums of the intervals in $\mathcal{U}_{i - 2}$ with the last interval of $\mathcal{U}_{i - 1}$, since all other sums have already been handled in prior steps. Thus, for each interval $I$ of $ulam\_ds$, we add it to the last interval, producing a pair of $DS$-subsets $one\_rep\_guess\_ds, mult\_rep\_guess\_ds$, where $one\_rep\_guess\_ds$ consists of all elements that can be written down as a sum in just one way, and $mult\_rep\_guess\_ds$ consists of all elements that can be written down as a sum in multiple ways---this is established in the proofs of correctness of Algorithms \ref{Sum Algorithm} and \ref{Sum Algorithm Redux}. We remove anything smaller than $largest\_computed + 2$ from both of these sets. Any element in $one\_rep\_ds$ that is either in $one\_rep\_guess\_ds$ or $mult\_rep\_guess\_ds$ is moved into $mult\_rep\_ds$, as we have shown that they are expressible as sums in multiple ways. We add to $mult\_rep\_ds$ anything in $mult\_rep\_guess\_ds$ as well. Any elements in $one\_rep\_guess\_ds$ that are not in $one\_rep\_ds$ or $mult\_rep\_ds$ are added to $one\_rep\_ds$, as they have not been found to be expressible as sums in more than one way.

As we go through every single sum in $\mathcal{U}_{i - 1} + \mathcal{U}_{i - 1}$ in this way, once we have cycled through every interval of $\mathcal{U}_{i - 1}$, $one\_rep\_ds$ (resp. $mult\_rep\_ds$) consists of all elements in $\mathcal{U}_{i - 1} + \mathcal{U}_{i - 1}$ larger than the largest element of $\mathcal{U}_{i - 1}$ that can be written as a sum of two distinct elements in $\mathcal{U}_{i - 1}$ in exactly one (resp. multiple) ways. Therefore, the smallest element $p$ of $one\_rep\_ds$ is an element of $U(1,X)$. We have to compute the largest element $q \in \ZZ[X]$ such that $[p,q] \in U(1,X)$. If the smallest interval of $one\_rep\_ds$ consists of more than one point, then that is the desired interval $[p,q]$. Otherwise, we note that $q < p + X$---otherwise, we would have that $q = (q - 1) + 1 = p + X$, which contradicts the definition of $U(1,X)$. Thus $q = p + X - 1$ unless there is an element in $one\_rep\_ds$ or $mult\_rep\_ds$ that is larger than $p$, but smaller than $p + X - 1$. This is a simple look-up, at the end of which we have computed the interval $[p,q]$ that we adjoin to $ulam\_ds$, giving $\mathcal{U}_i$. After updating $largest\_computed$, $one\_rep\_ds$, and $mult\_rep\_ds$, we are done.
\end{proof}

What relation does this algorithm have to Theorem \ref{Main Theorem}? To explain, it is helpful to first reformulate the result we wish to prove slightly. For any $n \in \ZZ$, let $\text{eval}_n: \ZZ[X] \rightarrow \ZZ$ be the evaluation homomorphism defined by $X \mapsto n$. Conjecture \ref{conj:rigidity} can be stated as
    \begin{align*}
        \text{eval}_n\left(U(1,X)\right) = U(1,n)
    \end{align*}
    
\noindent for all $n \geq 4$. Similarly, Theorem \ref{Main Theorem} is that there exists an algorithm $\mathcal{A}_\text{Ulam}$ such that on an input of $k > \NN$, it computes the coefficients $a_i,b_i,c_i,d_i$ of $U(1,X)$ and an $N \in \NN$ such that for all $n \geq N$,
    \begin{align*}
        \text{eval}_n\left(U(1,X) \cap [1,c_k n + d_k]\right) = U(1,n) \cap [1,c_k n + d_k].
    \end{align*}
    
\noindent We can now show how this follows from Algorithm \ref{Naive Algorithm}.

\begin{proof}[Proof of Theorem \ref{Main Theorem}]
Modify Algorithm \ref{Naive Algorithm} as follows: each time a comparison $p(X) < q(X)$, $p(X) \leq q(X)$, $p(X) > q(X)$, $p(X) \geq q(X)$, or $p(X) = q(X)$ is made, compute the smallest natural number $M$ such that for all $m \geq <$, replacing $X$ with $m$ does not change the truth value of the comparison---it is not hard to see that such an $<$ must exist by considering the behavior of these polynomials as $M \rightarrow \infty$. At the end of the computation, along with the coefficients $a_i,b_i,c_i,d_i$, return the maximum $N$ of all of the aforementioned $M$. Call this algorithm $\mathcal{A}_\text{Ulam}$.

Now, take any $n \geq N$ and consider modifying Algorithm \ref{Naive Algorithm} by replacing every set $S \subset \ZZ[X]$ used with $\text{eval}_n(S)$. It's easy enough to see that this will be an algorithm that computes the coefficents of $U(1,n)$, considered as a $DS$-subset. But since $N$ is chosen such that all comparisons of elements yield the same result as before, the output of this algorithm is the same as that of Algorithm \ref{Naive Algorithm}, which proves that
    \begin{align*}
        \text{eval}_n\left(U(1,X) \cap [1,c_k n + d_k]\right) = U(1,n) \cap [1,c_k n + d_k].
    \end{align*}
    
\noindent Thus, $\mathcal{A}_\text{Ulam}$ has the desired properties.
\end{proof}

The ideas of the construction of the algorithm $\mathcal{A}_\text{Ulam}$ can be generalized---there are in fact families of algorithms with its essential properties. The importance of this is that it might be possible to construct a more efficient algorithm using, for example, ideas of Gibbs and Judson \cite{Gibbs_2015,Gibbs_Judson_2015}. Before we show how to do this, we need some definitions.

\begin{definition}
Let $S$ be a computable ordered ring. We say that a formula is $S$-\emph{expressible} if it is either
	\begin{enumerate}
		\item any variable $x$,
		\item any constant $c \in S \cup \{\top, \bot\}$,
		\item any empty list $[]$,
	\end{enumerate}
	
\noindent or if can be constructed recursively from other $S$-expressible formulas according to the following rules.
	\begin{enumerate}
		\item If $\mathcal{E}_1, \mathcal{E}_2$ are $S$-expressible, then $\mathcal{E}_1 = \mathcal{E}_2$ is $S$-expressible.
		\item If $\mathcal{E}_1,\mathcal{E}_2$ are $S$-expressible, then $\mathcal{E}_1 + \mathcal{E}_2, \mathcal{E}_1 - \mathcal{E}_2, \mathcal{E}_1 \cdot \mathcal{E}_2, \mathcal{E}_1 / \mathcal{E}_2$ are $S$-expressible if defined.
		\item If $\mathcal{E}_1,\mathcal{E}_2$ are $S$-expressible, then $\mathcal{E}_1 < \mathcal{E}_2, \mathcal{E}_1 \leq \mathcal{E}_2, \mathcal{E}_1 > \mathcal{E}_2, \mathcal{E}_1 \geq \mathcal{E}_2$ are $S$-expressible if defined. (We call such formulas \emph{comparisons}.)
		\item If $\mathcal{E}_1,\mathcal{E}_2$ are $S$-expressible, then $\neg \mathcal{E}_1, \mathcal{E}_1 \lor \mathcal{E}_2, \mathcal{E}_1 \land \mathcal{E}_2, \mathcal{E}_1 \Rightarrow \mathcal{E}_2$ are $S$-expressible if defined.
		\item If $\mathcal{E}_1,\mathcal{E}_2,\ldots \mathcal{E}_n$ are $S$-expressible, then $\left[\mathcal{E}_1,\mathcal{E}_2,\ldots \mathcal{E}_n\right]$ is $S$-expressible.
	\end{enumerate}
	
\noindent We say that an algorithm $\mathcal{A}$ \emph{has basic steps over} $S$ if it is composed of the following basic steps:
	\begin{enumerate}
		\item Initializing $x \leftarrow \mathcal{E}$, where $x$ is a variable and $\mathcal{E}$ is $S$-expressible.
		\item Retrieving the $i$-th index of a list $l$.
		\item Checking an if-statement $\textbf{if}(\mathcal{E})$, where $\mathcal{E}$ is $S$-expressible, and branching accordingly.
		\item Running a while-loop $\textbf{while}(\mathcal{E})$, where $\mathcal{E}$ is $S$-expressible.
		\item Returning an $S$-expressible formula $\mathcal{E}$.
	\end{enumerate}
\end{definition}

For $\ZZ[X]$ and $\QQ[X]$, specifically, we also need a way to extend the evaluation homomorphism $\text{eval}_n$ to expressible formulas.

\begin{definition}
Let $\mathcal{E}$ be an $\QQ[X]$-expressible formula. We define $\text{eval}_n(\mathcal{E})$ to be the formula produced by replacing each instance of a constant $c$ in $\mathcal{E}$ with $\text{eval}_n(c)$.
\end{definition}

\noindent With these definitions, we can state the following result.

\begin{theorem}\label{Replacement Theorem}
Define $\mathcal{C}_\text{Ulam}$ to be the set of all algorithms $\mathcal{A}$ with basic steps over $\QQ[X]$ such that on an input of $k \in \NN$, $\mathcal{A}$ returns the first $k$ coefficients of $U(1,X)$ and a proof of correctness. There exists an algorithm such that, on an input of an algorithm $\mathcal{A} \in \mathcal{C}_\text{Ulam}$, it returns an algorithm which on an input of $k \in \NN$ it returns an $N \in \NN$ such that for all $n \geq N$,
    \begin{align*}
        \text{eval}_n\left(U(1,X) \cap [1,c_k n + d_k]\right) = U(1,n) \cap [1,c_k n + d_k].
    \end{align*}
\end{theorem}

\begin{proof}
The algorithm runs as follows: on an input of $k \in \NN$, run $\mathcal{A}(k)$, but each time a comparison $\mathcal{E}$ comes up in the computation, compute the smallest integer $M$ such that for all $m \geq M$, $\mathcal{E} = \text{eval}_m(\mathcal{E})$. Then return the maximum $N$ of all of these $M$---there will only be finitely many comparisons for any given input $k$.

Why does $N$ have claimed properties? For any $n \geq N$, define an algorithm $\mathcal{A}_n$ produced by replacing each $\QQ[X]$-expressible formula $\mathcal{E}$ in $\mathcal{A}$ with $\text{eval}_n(\mathcal{E})$. Since the truth value of comparisons is preserved and $\mathcal{A}$ has basic steps over $\QQ[X]$, $\mathcal{A}_n$ produces the same coefficients $a_i,b_i,c_i,d_i$ as $\mathcal{A}$, and verifies that if we define a subset
	\begin{align*}
	\mathcal{U}_n = \bigcup_{i = 0}^k \left[a_i n + b_i, c_i n + d_i\right],
	\end{align*}
	
\noindent then

	\begin{enumerate}
		\item $a_0 n + b_0 = c_0 n + d_0 = 1$ and $a_1 n + b_1 = n$,
		\item $a_{i + 1} n + b_{i + 1} > c_i n + d_i + 1$ for all $i \in \NN$, and
		\item for every $x \in \mathcal{U}_n$, $x$ is the smallest element of $\ZZ$ such that there exists exactly one way such that $x$ can be written as the sum of two distinct elements $y,z \in \mathcal{U}_n$.
	\end{enumerate}
	
\noindent In other words, $\mathcal{U}_n = U(1,n) \cap [1,c_k n + d_k] = \text{eval}_n(U(1,X) \cap [1,c_k X + d_k])$, ergo we have the desired result.
\end{proof}

\section{Numerical Results:}\label{section: numerical}
Algorithm \ref{Naive Algorithm} has in fact been implemented in Python by the author---using this algorithm, it is easy to show that for all $n \geq 10$
	\begin{align*}
	U(1,n) \cap [1,c_{150}n + d_{150}] = \bigcup_{i = 0}^{150} [a_i n + b_i, c_i n + d_i].
	\end{align*}
	
\noindent Unfortunately, this is substantially worse than what was formerly known. A slight improvement can be made by making use of previously gathered data, but not substantially. Nevertheless, this result suggests that it may be possible to prove a slightly weaker version of Conjecture \ref{conj:rigidity} by proving a result about the sort of comparisons that come up in Algorithm \ref{Naive Algorithm}. It may also be that there are better candidates in the class $\mathcal{C}_{\text{Ulam}}$ discussed in Theorem \ref{Replacement Theorem} than Algorithm \ref{Naive Algorithm}. After all, in practice, Algorithm \ref{Naive Algorithm} is not a particularly efficient method of computing the coefficients $a_i,b_i,c_i,d_i$ of $U(1,X)$---a naive implementation puts it in the $\Theta(k^2\log_2(k))$ complexity class, due to the two nested for-loops and the need to perform binary search for the set operations.

In practice, an easier approach toward computing the coefficients $a_i,b_i,c_i,d_i$ is to assume that Conjecture \ref{conj:rigidity} is true and that Ulam sequences grow linearly, and then to compute $U(1,4)$ and $U(1,5)$ up to a suitably large number of terms, from which one can compute the coefficients $a_i,b_i,c_i,d_i$. The results of this method can be proven correct after the fact by using Theorem 3.1 of \cite{HKSS_2019_2} and verifying that there exists a $B \approx 0.13901$ such that
	\begin{enumerate}
		\item $\left|b_i - B a_i\right|,\left|d_i - B c_i\right| < 2.5$ for all $i$, and
		\item for $n = 4,5,\ldots 14$,
			\begin{align*}
			U(1,n) \cap \left[1, c_k n + d_k\right] = \bigcup_{i = 0}^k \left[a_i n + b_i, c_i n + d_i\right].
			\end{align*}
	\end{enumerate}
	
\begin{figure}
\centering
\begin{tabular}{cc}
\includegraphics[width=0.45\textwidth]{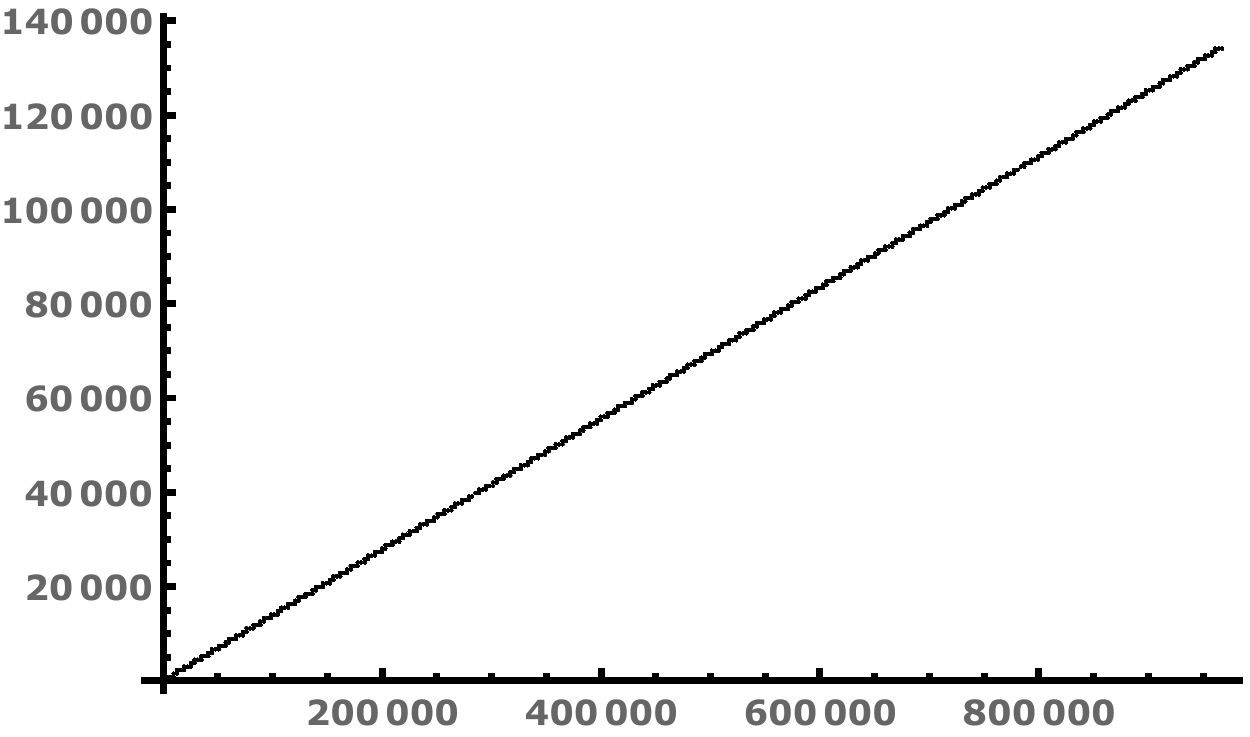} & \includegraphics[width=0.45\textwidth]{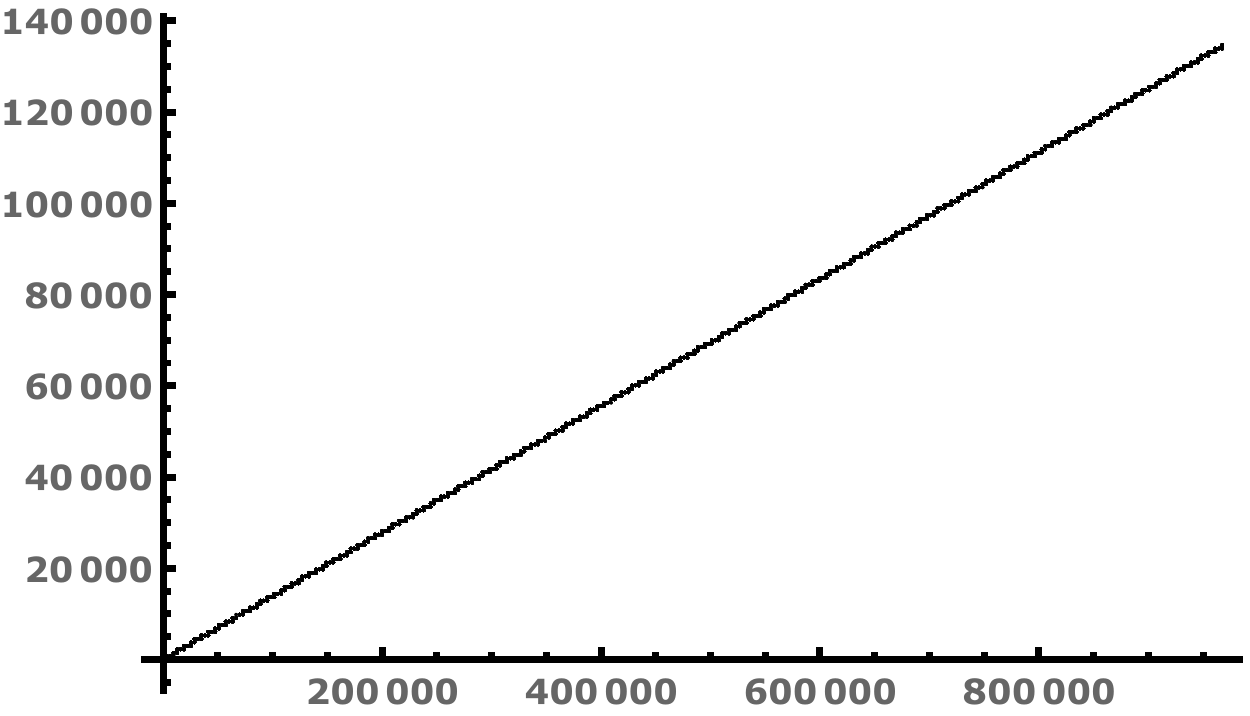}
\end{tabular}

\caption{A graph of $b_i$ as a function of $a_i$, and a graph of $d_i$ as a function of $c_i$ for $0 \leq i \leq 217529$.}
\label{Linear Regression Chart}
\end{figure}
	
\noindent This extraordinary linear dependence, originally noted in \cite{HKSS_2019_1}, is shown in Figure \ref{Linear Regression Chart}. In this way, using a basic $\Theta(k^2)$ algorithm for computing Ulam sequences $U(1,n)$, the author was able to compute coefficients $a_i,b_i,c_i,d_i$ such that for all $n \geq 4$,
	\begin{align*}
	U(1,n) \cap \left[1, c_{217529} n + d_{217529}\right] = \bigcup_{i = 0}^{217529} \left[a_i n + b_i, c_i n + d_i\right].
	\end{align*}
	
\noindent As $(a_{217529},b_{217529},c_{217529},d_{217529}) = (966409, 134342, 966410, 134340)$, this is the full list of coefficients $a_i,b_i,c_i,d_i$ such that $c_i n + d_i \leq 966410 n + 134340$. This data gives further numerical evidence for Conjecture \ref{conj:gibbslike}; specifically, defining $\lambda_n = 3n + 0.417031$, we consider the sets $U(1,n) \mod \lambda_n$. We find that $99.9\%$ of the terms smaller than $10^7 n$ lie in the interval $[\lambda_n/3,2\lambda_n/3]$---this is shown in Figures \ref{Middle Third Statistics} and \ref{Middle Third Diagram}.

\begin{figure}
\begin{tabular}{cccc}
$n$ & $\#$ of exceptions & $\#$ elements of $U(1,n)$ computed & Percent of exceptions \\ \hline
4 & 411 & 635045 & 0.0647198\% \\ 
5 & 416 & 814686 & 0.0510626\% \\
6 & 423 & 994573 & 0.0425308\% \\
7 & 426 & 1174402 & 0.0362738\% \\
8 & 427 & 1354386 & 0.0315272\% \\
9 & 430 & 1534323 & 0.0280254\% \\
10 & 430 & 1714404 & 0.0250816\% \\
11 & 432 & 1894463 & 0.0228033\% \\
12 & 433 & 2074581 & 0.0208717\% \\
13 & 436 & 2254856 & 0.019336\% \\
14 & 436 & 2435164 &  0.0179043\%
\end{tabular}

\caption{Computations of the number of elements $u \in U(1,n)$ such that $u \mod \lambda_n$ is not between $\lambda_n/3$ and $2\lambda_n/3$.}
\label{Middle Third Statistics}
\end{figure}

\begin{figure}
\centering
\begin{tabular}{cc}
\includegraphics[width=0.4\textwidth]{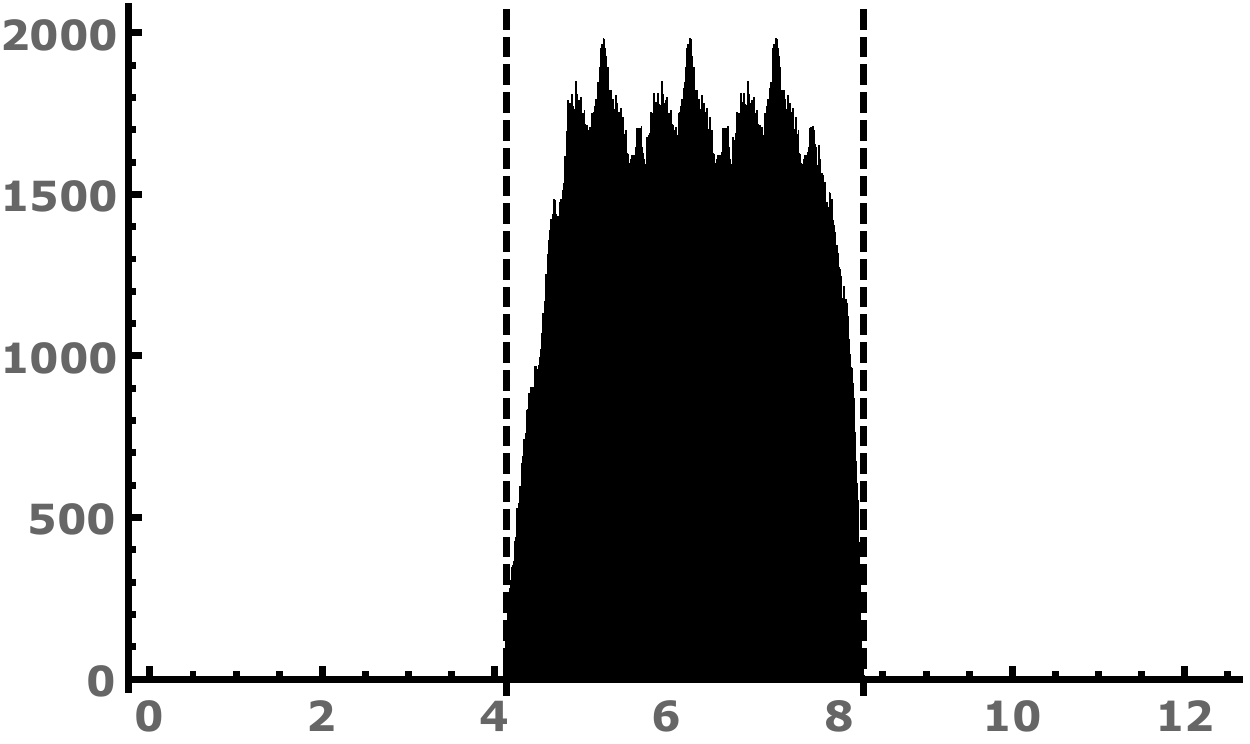} & \includegraphics[width=0.4\textwidth]{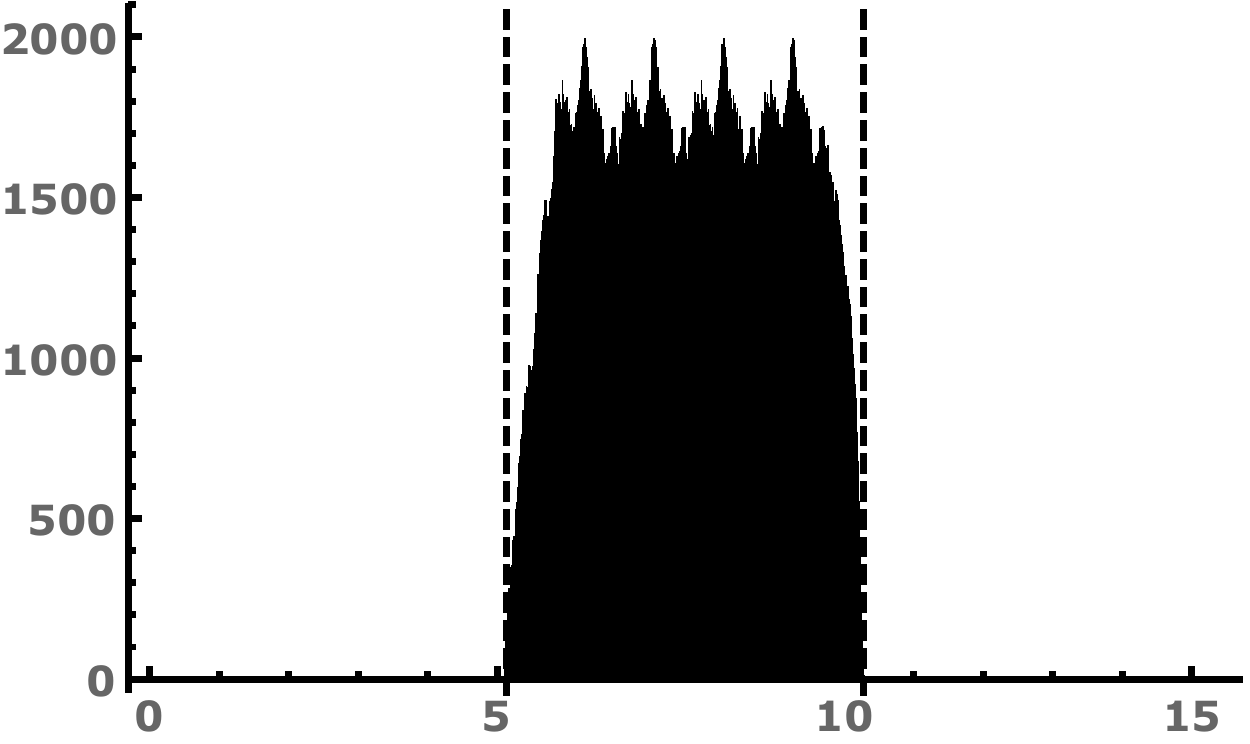} \\
\includegraphics[width=0.4\textwidth]{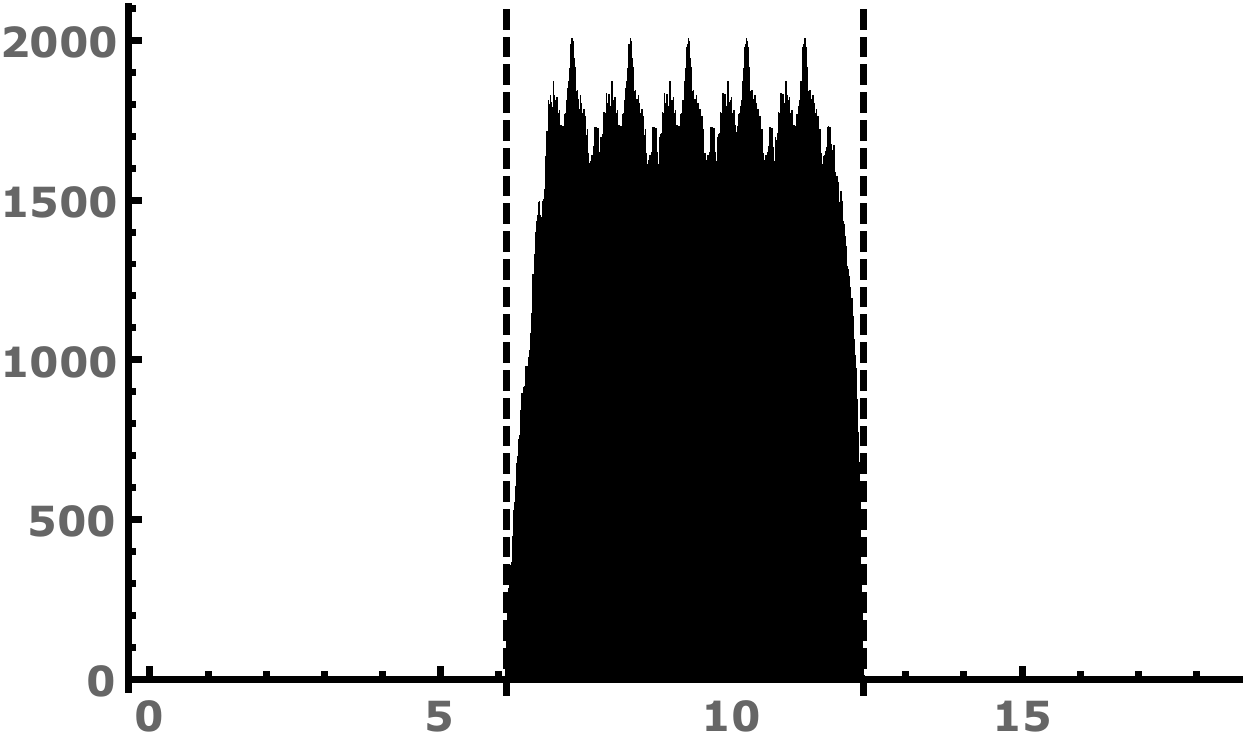} & \includegraphics[width=0.4\textwidth]{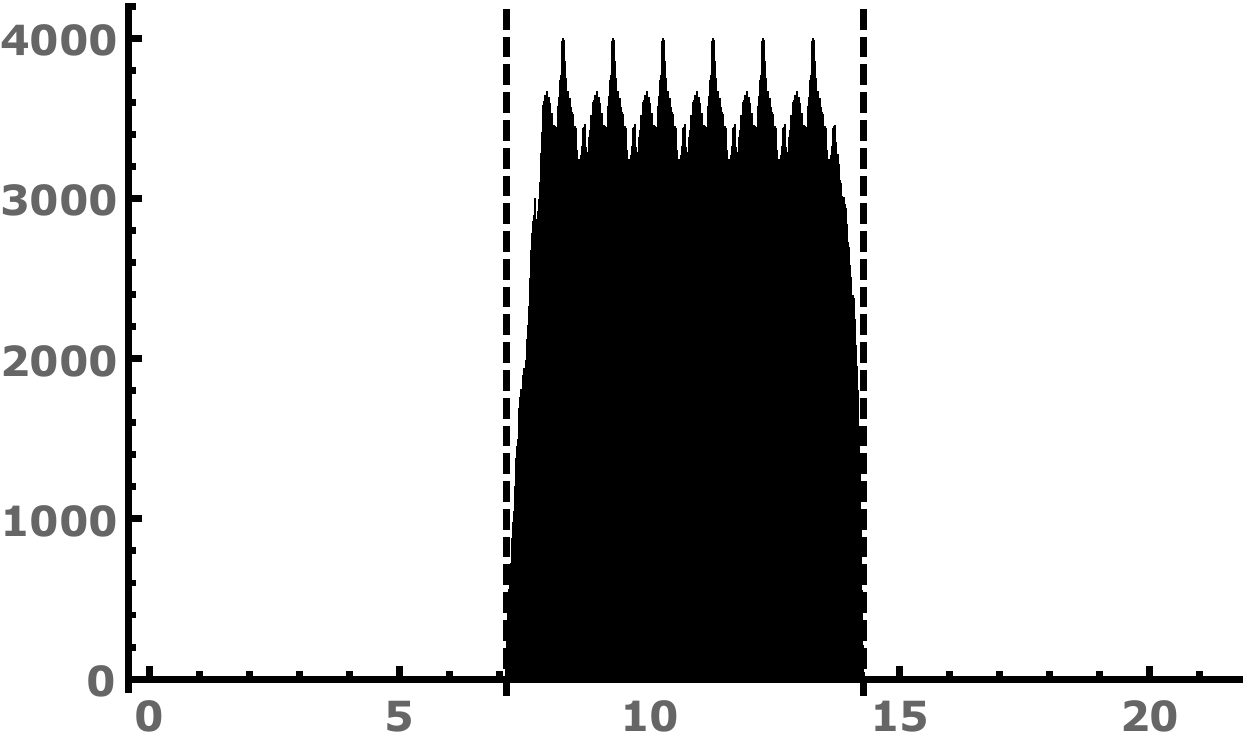} \\
\includegraphics[width=0.4\textwidth]{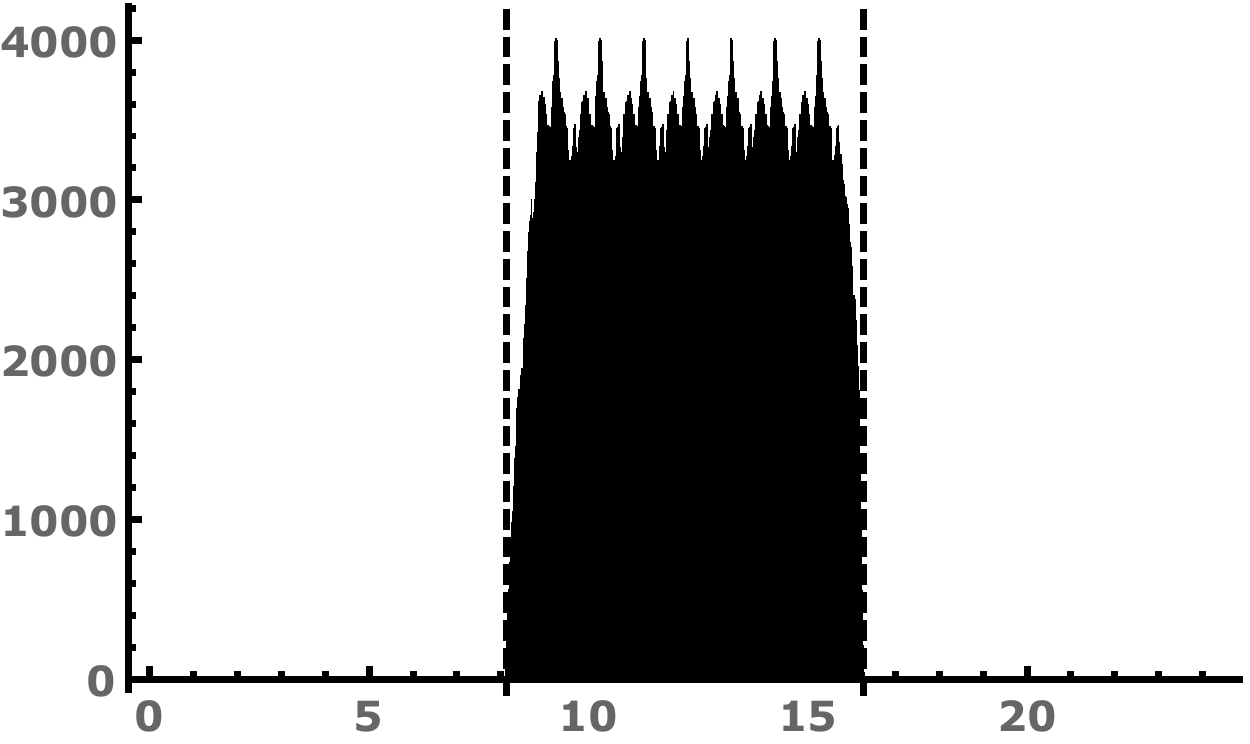} & \includegraphics[width=0.4\textwidth]{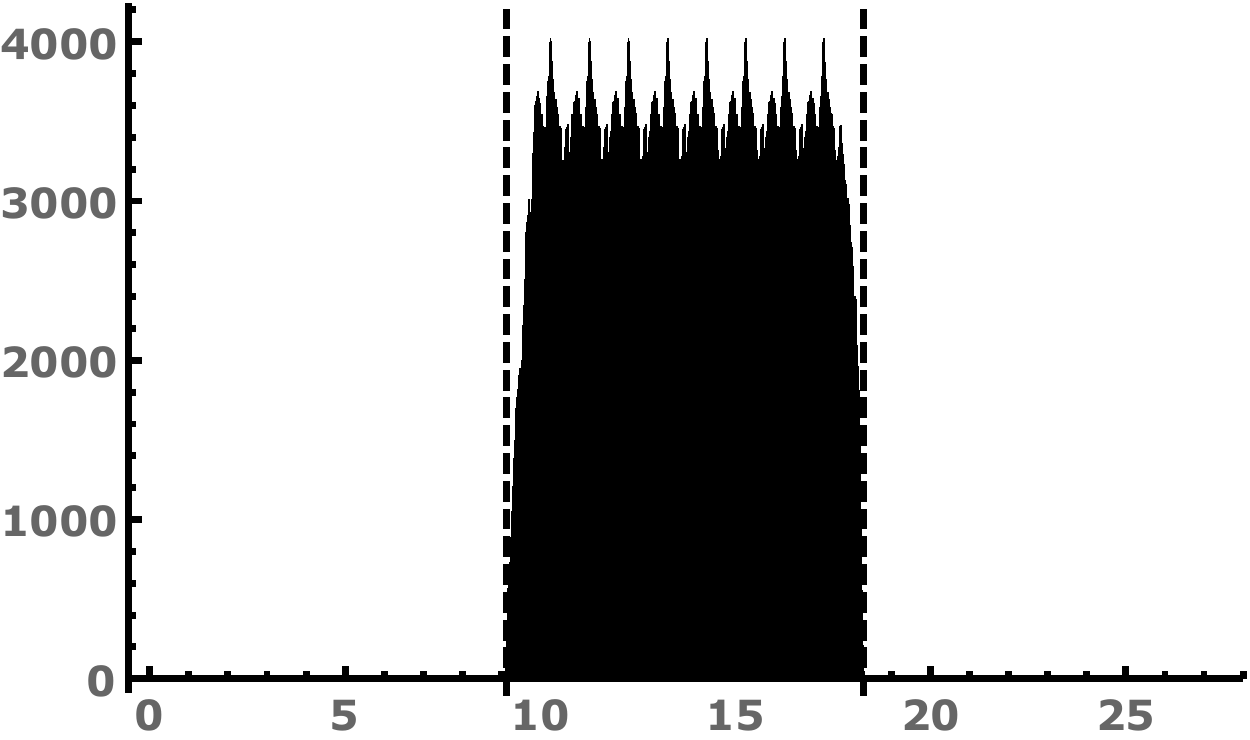}
\end{tabular}

\caption{From left to right, and top to bottom, histograms of $U(1,n) \mod \lambda_n$ for $n = 4,5,6,7,8,9$. The dashed lines split the interval into thirds.}
\label{Middle Third Diagram}
\end{figure}

With this motivation, it is natural to investigate whether there might be some ``magic polynomial" $\lambda(X)$ and a way to define $U(1,X) \mod \lambda(X)$ such that the resulting distribution has interesting properties. Not only is this possible, but the results are startling. First, we give a couple of definitions.
	\begin{definition}
	Given elements $p,q \in \RR[X]$, define their \emph{remainder set} to be
		\begin{align*}
		\mathcal{R}_{p,q} := \left\{p - sq\middle|s \in \ZZ[X], \ p - sq \geq 0\right\}.
		\end{align*}
		
	\noindent If $\mathcal{R}_{p,q}$ has a smallest element, we define $p \mod q = \min \mathcal{R}_{p,q}$.
	\end{definition}
	
\noindent Notice in particular that if we define $\lambda(X) = 3X + .417031$, then $U(1,X) \mod \lambda(X)$ is a well-defined subset of $\left\{a X + b\middle|a \in \{0,1,2\}, \ b \in \QQ\right\} \subset \QQ[X]$.

	\begin{definition}
	We say that an interval $[p,q] \subset \QQ[X]$ is $\emph{long}$ if $\deg(q - p) > 0$---otherwise, we say that the interval is \emph{short}.
	\end{definition}
	
\noindent Remarkably, long intervals in $U(1,X)$ seem to occur almost exactly four times as often as short ones---for the first $217530$ intervals, we find that $79.98\%$ are long, and $20.02\%$ are short. These two interval types seem to exhibit slightly different statistical behaviors modulo $\lambda(X)$, so we split into two cases accordingly.
	
\begin{description}
	\item[Long Intervals]
Let $\mathcal{U}_L$ be the subset of $U(1,X)$ consisting of all long intervals, and let $a_i',b_i',c_i',d_i'$ be the coefficients of $\mathcal{U}_L$. For all $i$ such that $c_i' X + d_i' \leq 966410 X + 134340$, we find that $a_i' \mod 3 = 1$ and $c_i' \mod 3 = 2$, hence $a_i'X + b_i' \mod \lambda(X) = X + \sigma_{i,1}$ and $c_i'X + d_i' \mod \lambda(X) = 2X + \sigma_{i,2}$. Furthermore, we find that $|\sigma_{i,1}|, |\sigma_{i,2}| < 3$ for all $i$. The statistical distributions of $\sigma_{i,1}$ and $\sigma_{i,2}$ are given in Figure \ref{Coefficient Distributions Long}.

\begin{figure}
\centering
\begin{tabular}{cc}
\includegraphics[width=0.4\textwidth]{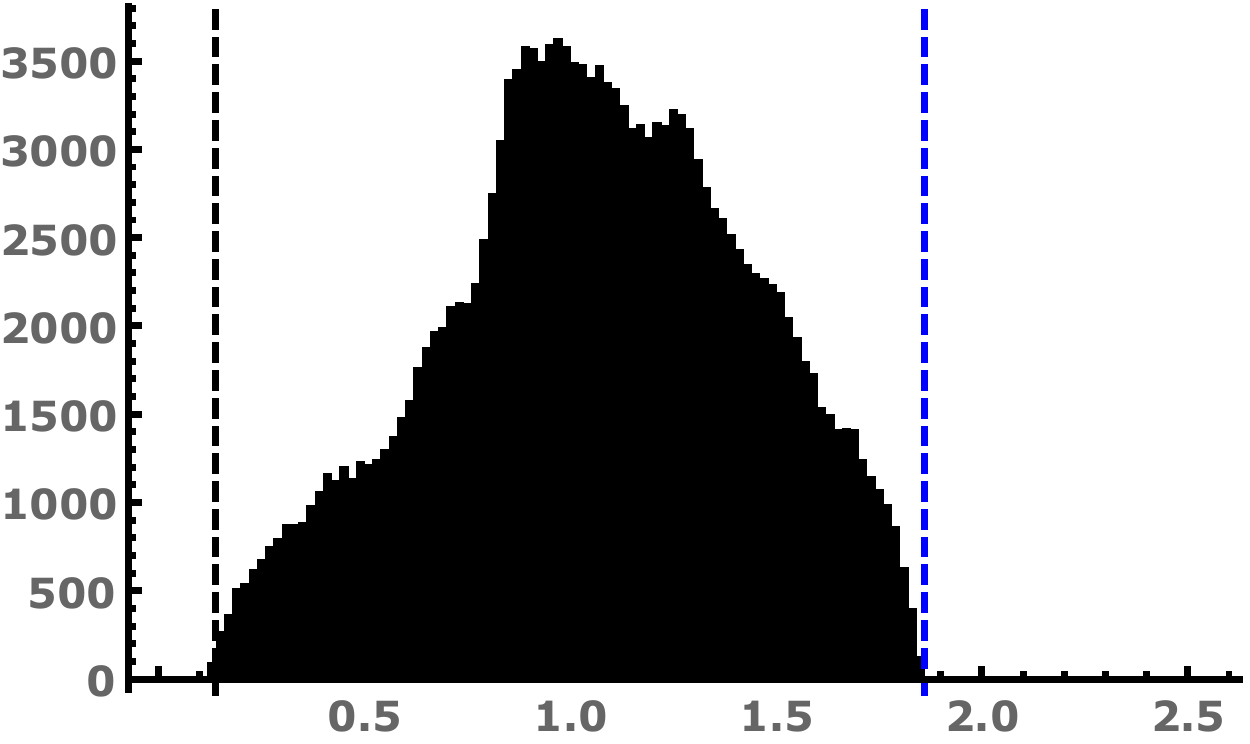} & \includegraphics[width=0.4\textwidth]{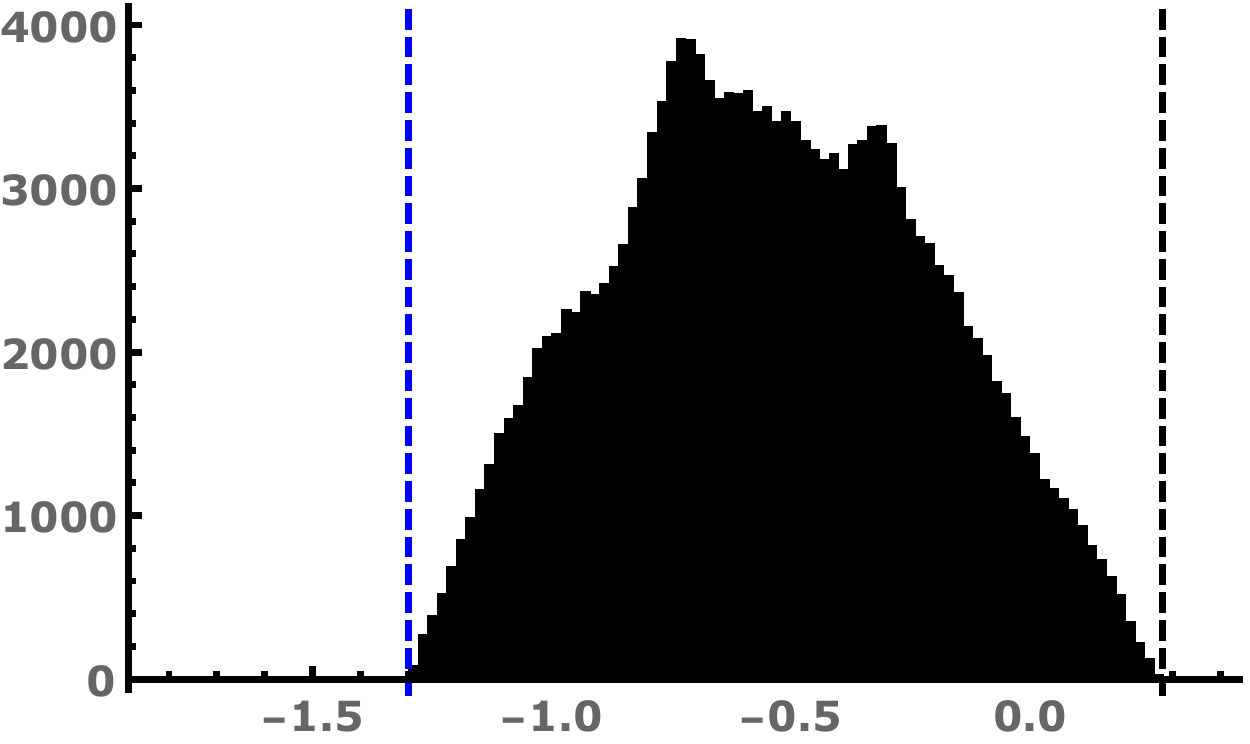}
\end{tabular}

\caption{Histograms of $\sigma_{i,1}$ (left) and $\sigma_{i,2}$ (right), together with black lines at $x = \lambda'/3$ (left) and $x = 2\lambda'/3$ (right), and blue lines at $x = 1.86$ (left) and $x = -1.3$ (right).}
\label{Coefficient Distributions Long}
\end{figure}

\item[Short Intervals]
Let $\mathcal{U}_S$ be the subset of $U(1,X)$ consisting of all short intervals, and let $a_i',b_i',c_i',d_i'$ be the coefficients of $\mathcal{U}_S$. For all $i$ such that $c_i' X + d_i' \leq 966410 X + 134340$, we find that $b_i' = d_i'$ with just a single exception, namely the interval $[25X + 4, 25X + 5]$. Furthermore, we find that $a_i' \mod 3 = 0$ if and only if $(a_i',b_i',c_i',d_i') = (0,1,0,1)$, $a_i' \mod 3 = 1$ is true for $\approx 43.7\%$ of indices $i$, and $a_i' \mod 3 = 2$ is true for $\approx 56.3\%$ of indices $i$. Therefore, $a_i'X + b_i' \mod \lambda(X) = X + \sigma_{i,3}$ or $a_i'X + b_i' \mod \lambda(X) = 2X + \sigma_{i,4}$, and we find that $|\sigma_{i,3}|,|\sigma_{i,4}| \leq 2$ for all $i$. The statistical distributions of $\sigma_{i,3}$ and $\sigma_{i,4}$ are given in Figure \ref{Coefficient Distributions Short}.

\begin{figure}
\centering
\begin{tabular}{cc}
\includegraphics[width=0.4\textwidth]{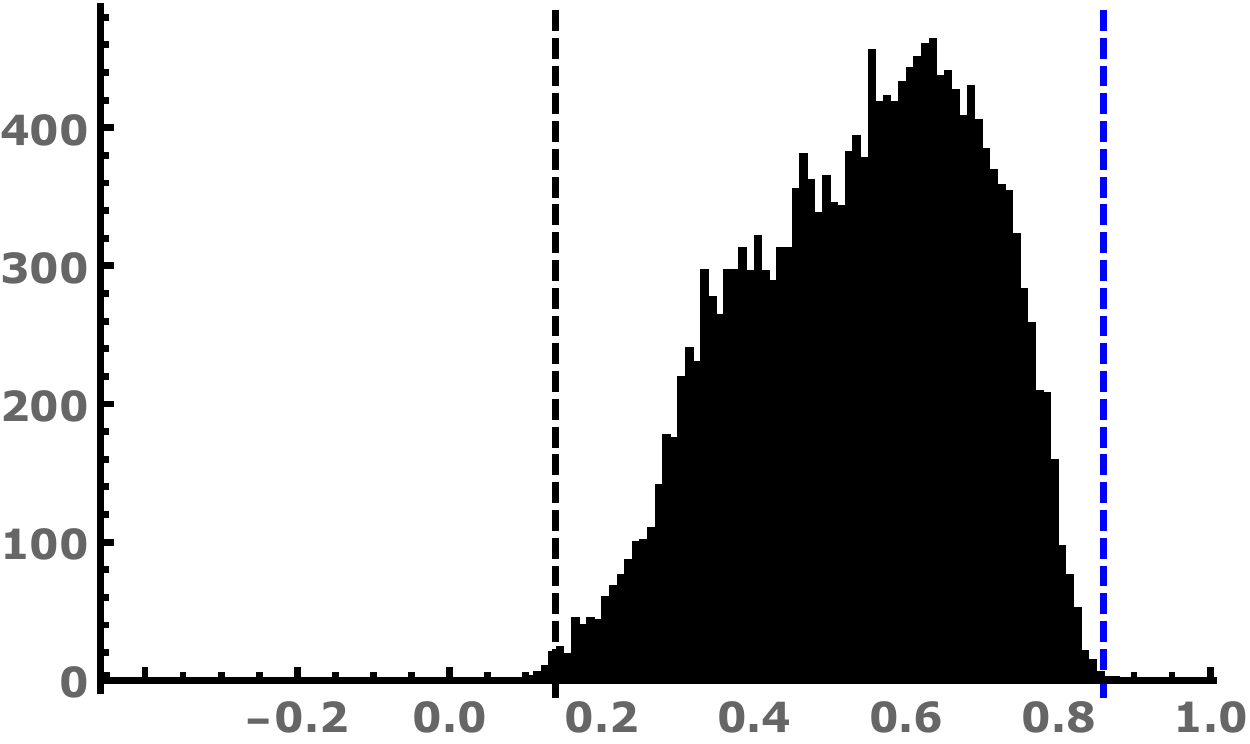} & \includegraphics[width=0.4\textwidth]{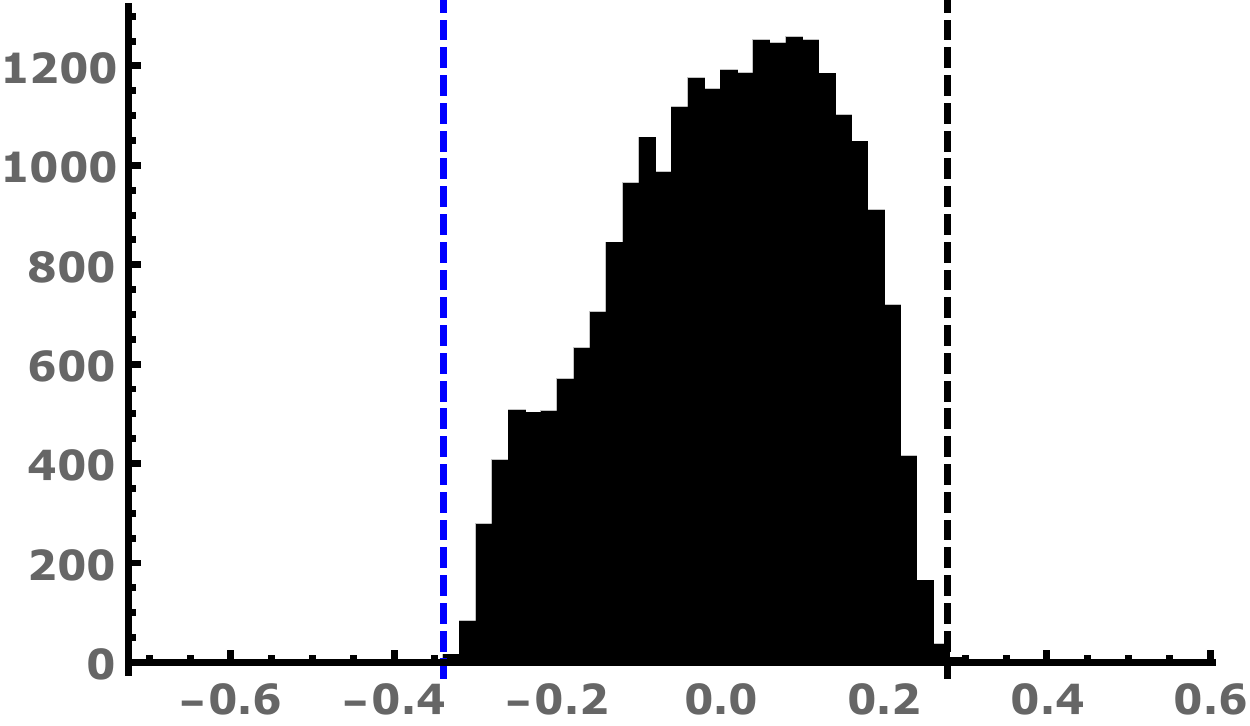}
\end{tabular}

\caption{Histograms of $\sigma_{i,3}$ (left) and $\sigma_{i,4}$ (right), together with black lines at $x = \lambda'/3$ (left) and $x = 2\lambda'/3$ (right), and blue lines at $x = 0.86$ (left) and $x = -0.34$ (right).}
\label{Coefficient Distributions Short}
\end{figure}
\end{description}

One final, but important numerical observation is that the bulk of of the $\sigma_{i,j}$ appear to be bounded---specifically, $\sigma_{i,1},\sigma_{i,3} \geq \lambda'/3$ and $\sigma_{i,2},\sigma_{i,4} \leq 2\lambda'/3$ for almost all $i$. Additionally, $\sigma_{i,1} \leq 1.86$, $\sigma_{i,3} \leq 0.86$, $\sigma_{i,2} \geq -1.3$, and $\sigma_{i,4} \geq -0.34$ for almost all $i$. Taking all of this information together, we precisely come up with Conjecture \ref{conj:stat}. Note additionally that $1.86 \approx 2 - \lambda'/3$ and $-1.3 \approx -1 - 2\lambda'/3$; unfortunately, the numerical evidence is not strong enough to conjecture this with any strong degree of certainty.

\section{Relations Between Conjectures:}\label{section: conjectures}
The importance of Conjecture \ref{conj:stat} is that, if true, then it sheds light on other open questions about Ulam sequences. We give a few examples.

\begin{theorem}
If Conjecture \ref{conj:stat} is true, then taking $B = \lambda'/3$, for all $\epsilon > 0$, if $i$ is sufficiently large,
	\begin{align*}
	\left|b_i - B a_i\right|,\left|d_i - B c_i\right| < \min\left(\sigma_1 -\frac{\lambda'}{3} + \epsilon, \frac{2\lambda'}{3} - \sigma_2 + \epsilon\right).
	\end{align*}
\end{theorem}

\begin{proof}
Note that if $i > 0$, then
	\begin{align*}
	a_iX + b_i \mod 3X + \lambda' &= a_iX + b_i - (3X + \lambda')\lfloor a_i/3\rfloor \\
	&= \left(a_i - 3\lfloor a_i/3\rfloor\right)X + b_i + \frac{\lambda'}{3}\left(a_i - \lfloor a_i/3\rfloor\right) - \frac{\lambda' a_i}{3} \\
	&= \begin{cases} X + b_i + \frac{\lambda'}{3} - \frac{\lambda' a_i}{3} & \text{if } a_i \mod 3 = 1 \\
	2X + b_i + \frac{2\lambda'}{3} - \frac{\lambda' a_i}{3} & \text{if } a_i \mod 3 = 2. \end{cases}
	\end{align*}
	
\noindent Therefore, for sufficiently large $i$, either
	\begin{align*}
	b_i + \frac{\lambda'}{3} - \frac{\lambda' a_i}{3} &\in \left(\frac{\lambda'}{3} - \epsilon, \sigma_1 + \epsilon\right) \\
	b_i - \frac{\lambda' a_i}{3} &\in \left(-\epsilon, \sigma_1 -\frac{\lambda'}{3} + \epsilon\right) \\
	\left|b_i - \frac{\lambda' a_i}{3}\right| &< \sigma_1 -\frac{\lambda'}{3} + \epsilon
	\end{align*}
	
\noindent or
	\begin{align*}
	b_i + \frac{2\lambda'}{3} - \frac{\lambda' a_i}{3} &\in \left(\sigma_2 - \epsilon, \frac{2\lambda'}{3} + \epsilon\right) \\
	b_i - \frac{\lambda' a_i}{3} &\in \left(\sigma_2 - \frac{2\lambda'}{3} - \epsilon, \epsilon\right) \\
	\left|b_i - \frac{\lambda' a_i}{3}\right| &< \frac{2\lambda'}{3} - \sigma_2 + \epsilon.
	\end{align*}
	
\noindent The argument for $c_i X + d_i$ is identical.
\end{proof}

\begin{theorem}
If Conjecture \ref{conj:rigidity} and Conjecture \ref{conj:stat} are true, then Conjecture \ref{conj:gibbslike} is true for $n \geq 4$.
\end{theorem}

\begin{proof}
Fix an $n \geq 4$ and define $\lambda_n = 3n + \lambda'$. Choose an $\epsilon > 0$. By Conjecture \ref{conj:stat}, we get that if $i$ is sufficiently large, then
	\begin{align*}
	a_in + b_i - \lambda_n \left\lfloor \frac{a_i}{3}\right\rfloor &= \left(a_i - 3\left\lfloor \frac{a_i}{3}\right\rfloor\right)n + \left(b_i - \lambda'\left\lfloor \frac{a_i}{3}\right\rfloor\right) \\
	&\in \left(n + \frac{\lambda'}{3} - \epsilon, n + \sigma_1 + \epsilon\right) \cup \left(2n + \sigma_2 - \epsilon, 2n + \frac{2\lambda'}{3} + \epsilon\right) \\
	&\in \left(\frac{\lambda_n}{3} - \epsilon, \frac{2\lambda_n}{3} + \epsilon\right).
	\end{align*}
	
\noindent Since this gives a real number between $0$ and $\lambda_n$, we conclude that for $i$ sufficiently large,
	\begin{align*}
	a_in + b_i \mod \lambda_n \in &\left(\frac{\lambda_n}{3} - \epsilon, \frac{2\lambda_n}{3} + \epsilon\right),
	\end{align*}
	
\noindent and similarly for $c_in + d_i$---thus, by Conjecture \ref{conj:rigidity}, it shall suffice to prove that for all $a_i n + b_i < u < c_i n + d_i$,
	\begin{align*}
	u \mod \lambda_n \in &\left(\frac{\lambda_n}{3} - \epsilon, \frac{2\lambda_n}{3} + \epsilon\right).
	\end{align*}
	
\noindent If $a_i = b_i$, then $u = a_i n + r$, where $b_i < r < d_i$. Therefore,
	\begin{align*}
	u - \lambda_n \left\lfloor \frac{u}{3}\right\rfloor	&= \left(a_i - 3\left\lfloor \frac{a_i}{3}\right\rfloor\right)n + \left(r - \lambda'\left\lfloor \frac{a_i}{3}\right\rfloor\right) \\
	&\in \left(\frac{\lambda_n}{3} - \epsilon, \frac{2\lambda_n}{3} + \epsilon\right),
	\end{align*}
	
\noindent hence we have the desired conclusion. If $a_i \neq b_i$, then $b_i = a_i + 1$, and therefore $(a_i,b_i) \mod 3 = (1,2)$. In that case, we know that
	\begin{align*}
	a_i n + b_i \mod \lambda_n &\in \left(n + \frac{\lambda'}{3} - \epsilon, n + \sigma_1 + \epsilon\right).
	\end{align*}
	
\noindent Since $u - (a_i n + b_i) < n$ and $2n + \sigma_1 + \epsilon < \lambda_n$ if $\epsilon$ is small enough, we conclude that
	\begin{align*}
	u \mod \lambda_n = \left(a_i n + b_i \mod \lambda_n\right) + u - a_i n + b_i,
	\end{align*}
	
\noindent whence
	\begin{align*}
	u \mod \lambda_n \in &\left(\frac{\lambda_n}{3} - \epsilon, \frac{2\lambda_n}{3} + \epsilon\right).
	\end{align*}
\end{proof}

\bibliography{UlamLibrary3}
\bibliographystyle{alpha}
\end{document}